\documentclass[reqno,11pt]{amsart}

\usepackage[a4paper,left=3cm, hmarginratio=1:1, top=3cm, bottom=3cm]{geometry}
\usepackage{graphicx}

\usepackage{amsmath}        
\usepackage{amssymb}
\usepackage{amsthm}
\usepackage[nooneline,hang]{subfigure}
\usepackage[font=normalsize,singlelinecheck=off,labelfont=bf]{caption}
\usepackage{graphicx}
\usepackage{bbold}
\usepackage[usenames]{color}
\usepackage{dsfont}
\usepackage{floatflt}
\usepackage{enumerate}
\usepackage{mathrsfs}
\usepackage{url}
\usepackage{appendix}

\usepackage{tikz}
\usetikzlibrary{arrows}

\newcommand{\R}{\mathds R}

\newcommand{\N}{\mathds N}
\newcommand{\Z}{\mathds Z}

\newcommand{\sn}{\text{sn}}
\newcommand{\cn}{\text{cn}}
\newcommand{\dn}{\text{dn}}
\renewcommand{\O}{\mathcal O}
\newcommand{\eps}{\varepsilon}

\newcommand{\sign}{\text{sign}}
\renewcommand{\u}{u}
\newcommand{\pu}{\mathcal U}

\newtheoremstyle{plainNoItalics}{}{}{\normalfont}{}{\bfseries}{.}{ }{}

\theoremstyle{plain}
\newtheorem{theorem}{Theorem}[section]
\newtheorem{proposition}[theorem]{Proposition}

\newtheorem{lemma}[theorem]{Lemma}

\theoremstyle{plainNoItalics}
\newtheorem{remark}[theorem]{Remark}

\newtheorem*{theorem*}{Theorem}
\newtheorem*{proposition*}{Proposition}
\newtheorem*{lemma*}{Lemma}
\newtheorem*{corollary*}{Corollary}
\newtheorem*{remark*}{Remark}

\newtheorem*{observation*}{Observation}
\newtheorem*{example*}{Example}
\newtheorem*{examples*}{Examples}
\newtheorem*{assumption*}{Assumption}

\theoremstyle{definition}
\newtheorem{definition}[theorem]{Definition}
\newtheorem*{definition*}{'Definition'}

\newtheorem*{definitionu*}{Definition}

\title[Approximation of sinh-Gordon solution via hyperbolic
orthogonal ring patterns]{Approximation of solutions of the sinh-Gordon equation $\Delta u 
-\sinh(2u)=0$ by hyperbolic orthogonal ring patterns}
\author{Ulrike B\"ucking}

\begin{document}

\date{\today}

\begin{abstract}
We consider hyperbolic orthogonal ring patterns as introduced in~\cite{Bo25} and focus on their characterization by uniformizing variables at the centers of the rings. Given a smooth solution of the sinh-Gordon equation $\Delta u 
-\sinh(2u)=0$, we restrict to a compact subset of its domain and discretize it by square grid lattices with edge length $\eps$. Taking the values of $u$ as Dirichlet boundary conditions, we prove that the corresponding uniformizing variables $u^\eps$ of the hyperbolic ring patterns converge to $u$ in $C^\infty$ with error of order $\eps^2$, given that the pairs of rings suitably converge to circles.
As a consequence we deduce that the hyperbolic orthogonal ring patterns converge to a harmonic map to the hyperbolic plane.
\end{abstract}

\maketitle

\section{Introduction}

A hyperbolic ring is a pair of two 
concentric circles in the hyperbolic plane $H^2$, that form a ring (annulus).
We consider patterns of rings in the hyperbolic plane such that neighboring rings intersect 
orthogonally, i.e., the larger circle of one ring intersects the smaller circle 
of the other orthogonally and vice versa, see Figure~\ref{fig:Ring}~left. In a pattern with the 
combinatorics of the square grid, this orthogonality implies that in one diagonal direction the two 
smaller circles and in the other diagonal direction the two larger circles touch at this point.
Such hyperbolic orthogonal ring patterns have been introduced and studied by 
Bobenko in~\cite{Bo25}. Orthogonal ring patterns may also be defined in 
spherical geometry, see \cite{Bo25}, and in Euclidean geometry, see~\cite{BHR24}. These patterns can 
be characterized by uniformizing parameters  associated to the centers of the rings. 

Orthogonal ring patterns can be considered as 
generalizations of orthogonal circle patterns which have in particular been studied in the context of discrete conformal mappings.
One important question in the theory of discrete conformal maps concerns convergence and approximation issues.
For  planar orthogonal circle patterns with combinatorics of the square grid,
convergence to a smooth conformal map has first been shown in~\cite{Sch97} for the case when the radii of the boundary
circles are given according to the derivative of a conformal map (more precisely by the logarithm of its modulus). 

As orthogonal ring patterns share similar properties 
as orthogonal ring patterns, it is a natural goal to prove 
convergence to suitable smooth functions. As explained 
in~\cite[Section~4]{Bo25}, the uniformizing parameters of hyperbolic orthogonal ring pattern at a given point infinitesimally (i.e.\ in the limit of small smoothly varying rings) approximate smooth functions $u$ which satisfy the sinh-Gordon equation $\Delta
u -\sinh(2u)=0$.  In this article we assume given such a smooth function $u$. We approximate a compact subset 
of its domain by parts of a square grid lattice with edge length $\eps>0$.
Reading off the values of $u$ at boundary vertices as Dirichlet boundary condition for the uniformizing variables $u^\eps$, corresponding hyperbolic 
orthogonal ring patterns exist by~\cite[Theorem~6.1]{Bo25}. 
We prove that $u^\eps$ uniformly approximate the values of the given function $u$ at the 
corresponding vertices with an error of order $\eps^2$. Furthermore, we show 
that the convergence even is in $C^\infty$ thanks to the regular structure of  the square grid lattice. 

The convergence of the uniformizing variables $u^\eps$ to $u$ implies that the 
hyperbolic ring patterns as mappings to $H^2$ converge to a harmonic map to $H^2$. See for example~\cite{He82,FD22} for a connection between such harmonic maps and solutions of the sinh-Gordon equation $\Delta u -\sinh(2u)=0$. As explained for example in~\cite{BJS19},
harmonic maps to $H^2$ are Gauss maps of CMC-surfaces in Lorentz space $\R^{2,1}$ and the function $u$ represents the conformal metric $\text{e}^{2u}$. Note that 
discrete CMC-surfaces in Lorentz space $\R^{2,1}$ can be constructed using hyperbolic ring 
patterns, see~\cite[Sections 11\&12]{BHS24}. Thus, the family of hyperbolic ring patterns for $u^\eps$ corresponds to a family of associated $S_1$-cmc pairs. 

The paper is organized as follows. In Sections~\ref{sec:not} and 
\ref{sec:char}, we introduce some notation 
and recall from~\cite{Bo25} useful characterizations and properties of hyperbolic orthogonal ring patterns. We focus in 
particular on the uniformizing variables associated to the centers of the rings,
which characterize these patterns. In Section~\ref{Sec:Convu}, we formulate and prove our theorem on $C^\infty$-convergence for these uniformizing variables 
and deduce the convergence for hyperbolic orthogonal ring patterns in Section~\ref{Sec:ConvR}.

\section*{Acknowledgements}
The author warmly thanks Alexander I.~Bobenko and Nina Smeek for 
insightful conversations on orthogonal ring patterns.

\section{Notation}\label{sec:not}
The following notation will be used throughout this article.

Let $\mathscr D\subset \Z^2$ be a cell complex defined by a subset of  elementary squares of the lattice $\Z^2$ and assume that  $\mathscr D$ is 
simply connected. The vertices are indexed by $(m,n)\in\Z^2$. In order to 
distinguish centers of rings and their intersection points, we split the 
vertices of $\mathscr D$ in two subsets: those vertices $v_{m,n}$  with ``odd'' index $m+n\equiv 1 \pmod 2$ will be associated to intersection 
(touching) points of circles$t_{m,n}$ whereas vertices $v_{m,n}$ with ``even'' index $m+n\equiv 0 \pmod 2$ will be identified with centers $k_{m,n}\in H^2$ of rings. To these two sets of
vertices we associate two planar graphs $G$ and $G^*$, which are dual to each other, as follows.
The vertices $V(G)$ are all even vertices of 
$\mathscr D$. The edges $E(G)$ correspond to faces of $\mathscr D$, that is
two vertices of $G$ are connected by an edge if and only if they
are incident to the same face of $\mathscr D$. The
dual graph $G^*$ is constructed analogously by taking for
$V(G^*)$ all odd vertices of $\mathscr D$.

For every ring we denote the inner circle and its hyperbolic radius by small 
letters $c$ and $r$ and 
the outer circle and its hyperbolic radius by capital letters $C$ and $R$. The 
outer 
radius will always be positive. By allowing $r$ to be negative, an orientation 
is assigned to the rings: a positive radius $r$ corresponds to counterclockwise 
orientation and a negative radius $r$ to clockwise orientation. As for vertices, we use
subscripts to associate circles and radii to vertices of $G$.

\section{Characterization of hyperbolic orthogonal ring 
patterns}\label{sec:char}
We consider hyperbolic orthogonal ring patterns with the combinatorics of the square grid as defined in~\cite{Bo25}. In this section, we gather definitions and properties as a basis for the convergence proof.

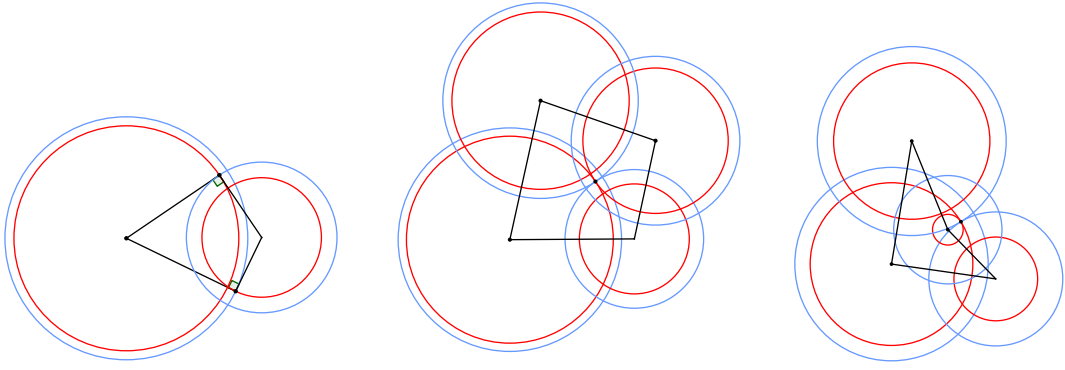
\begin{figure}
\definecolor{wwzzff}{rgb}{0.4,0.6,1}
\definecolor{ffqqqq}{rgb}{1,0,0}
\definecolor{qqwuqq}{rgb}{0,0.39215686274509803,0}
\definecolor{bubtba}{rgb}{0.7,0.7,0.7294}
\begin{tikzpicture}[line cap=round,line join=round,>=triangle 45,x=1cm,y=1cm, 
scale=0.25]
\clip(-17.94,-7.58) rectangle (2.7,7.58);
\draw[line width=0.5pt,color=qqwuqq,fill=qqwuqq,fill 
opacity=0.10000000149011612] 
(-5.035034500214293,-2.4181783611566474) -- 
(-5.4168561390576455,-2.2332128613709403) -- 
(-5.601821638843353,-2.6150345002142927) -- (-5.22,-2.8) -- cycle; 
\draw[line width=0.5pt,color=qqwuqq,fill=qqwuqq,fill 
opacity=0.10000000149011612] 
(-6.431020875765418,3.1017053404356716) -- 
(-6.192726216201089,2.750684464670254) -- 
(-5.841705340435672,2.9889791242345822) -- (-6.08,3.34) -- cycle; 
\draw [line width=0.5pt] (-11,0)-- (-5.22,-2.8);
\draw [line width=0.5pt] (-11,0)-- (-6.08,3.34);
\draw [line width=0.5pt,color=ffqqqq] (-11,0) circle (5.946595664748024cm);
\draw [line width=0.5pt,color=wwzzff] (-11,0) circle (6.422491728293622cm);
\draw [line width=0.5pt,color=wwzzff] (-3.842361945757712,0.04383855482872404) 
circle (3.9839307642611974cm);
\draw [line width=0.5pt,color=ffqqqq] (-3.842361945757712,0.04383855482872404) 
circle (3.1599532171262914cm);
\draw [line width=0.5pt] (-6.08,3.34)-- 
(-3.842361945757712,0.04383855482872404);
\draw [line width=0.5pt] (-3.842361945757712,0.04383855482872404)-- 
(-5.22,-2.8);
\begin{scriptsize}
\draw [fill=black] (-11,0) circle (2.5pt);
\draw [fill=black] (-5.22,-2.8) circle (2.5pt);
\draw [fill=black] (-6.08,3.34) circle (2.5pt);
\draw [fill=black] (-3.842361945757712,0.04383855482872404) circle (0.5pt);
\end{scriptsize}
\end{tikzpicture}
\begin{tikzpicture}[line cap=round,line join=round,>=triangle 45,x=1cm,y=1cm, 
scale=0.23]
\clip(-17.6,-8.16) rectangle (4.36,14.58);
\draw [line width=0.5pt,color=ffqqqq] (-11,0) circle (5.946595664748024cm);
\draw [line width=0.5pt,color=wwzzff] (-11,0) circle (6.422491728293622cm);
\draw [line width=0.5pt,color=wwzzff] (-3.842361945757712,0.04383855482872404) 
circle (3.9839307642611974cm);
\draw [line width=0.5pt,color=ffqqqq] (-3.842361945757712,0.04383855482872404) 
circle (3.1599532171262914cm);
\draw [line width=0.5pt,color=ffqqqq] (-2.6230885130931494,5.686765115095301) 
circle (4.178222532816651cm);
\draw [line width=0.5pt,color=wwzzff] (-9.2348747242803,7.987300492053616) 
circle (5.616995316834836cm);
\draw [line width=0.5pt,color=ffqqqq] (-9.2348747242803,7.987300492053616) 
circle (5.0957697542131255cm);
\draw [line width=0.5pt,color=wwzzff] (-2.6230885130931494,5.686765115095301) 
circle (4.848445226453897cm);
\draw [line width=0.5pt] (-11,0)-- (-3.842361945757712,0.04383855482872404)-- 
(-2.6230885130931494,5.686765115095301)-- (-9.2348747242803,7.987300492053616);
\draw [line width=0.5pt] (-9.2348747242803,7.987300492053616)-- (-11,0);
\begin{scriptsize}
\draw [fill=black] (-11,0) circle (2.5pt);
\draw [fill=black] (-6.08,3.34) circle (2.5pt);
\draw [fill=black] (-3.842361945757712,0.04383855482872404) circle (0.5pt);
\draw [fill=black] (-2.6230885130931494,5.686765115095301) circle (2.5pt);
\draw [fill=black] (-9.2348747242803,7.987300492053616) circle (2.5pt);
\end{scriptsize}
\end{tikzpicture}
\begin{tikzpicture}[line cap=round,line join=round,>=triangle 45,x=1cm,y=1cm, 
scale=0.2]
\clip(-17.6,-7.23) rectangle (4.36,15.51);
\draw [line width=0.5pt,color=ffqqqq] (-10.67,0.54) circle 
(5.376625335654326cm);
\draw [line width=0.5pt,color=wwzzff] (-10.67,0.54) circle 
(6.392034105040429cm);
\draw [line width=0.5pt,color=wwzzff] 
(-3.7739033559328683,-0.44035128438147647) circle (4.428220586770985cm);
\draw [line width=0.5pt,color=ffqqqq] 
(-3.7739033559328683,-0.44035128438147647) circle (2.7675146910364123cm);
\draw [line width=0.5pt,color=ffqqqq] (-6.945223177422092,2.812195011594367) 
circle (1.0135034546237376cm);
\draw [line width=0.5pt,color=wwzzff] (-9.331395372152194,8.669965985063774) 
circle (6.243405262193826cm);
\draw [line width=0.5pt,color=ffqqqq] (-9.331395372152194,8.669965985063774) 
circle (5.166921092662287cm);
\draw [line width=0.5pt,color=wwzzff] (-6.945223177422092,2.812195011594367) 
circle (3.5798615358879218cm);
\draw [line width=0.5pt] (-10.67,0.54)-- 
(-3.7739033559328683,-0.44035128438147647)-- 
(-6.945223177422092,2.812195011594367)-- (-9.331395372152194,8.669965985063774);
\draw [line width=0.5pt] (-9.331395372152194,8.669965985063774)-- (-10.67,0.54);
\begin{scriptsize}
\draw [fill=black] (-10.67,0.54) circle (2.5pt);
\draw [fill=black] (-6.08,3.34) circle (2.5pt);
\draw [fill=black] (-3.7739033559328683,-0.44035128438147647) circle (0.5pt);
\draw [fill=black] (-6.945223177422092,2.812195011594367) circle (2.5pt);
\draw [fill=black] (-9.331395372152194,8.669965985063774) circle (2.5pt);
\end{scriptsize}
\end{tikzpicture}
 \caption{Left: Two orthogonally intersecting rings. Middle: 
The inner circles touch along one dirction and the outer 
circles touch at the same point along the other direction. 
Right: If the orientation (i.e.\ sign of radii) of the inner 
circles differ, the centers lie on the same side of the common 
tangent.}\label{fig:Ring}
\end{figure}

\begin{definition}[{Orthogonal hyperbolic ring patterns, see~\cite[Def.~2.1 and Sec.~3]{Bo25}}]\label{def:Ringpattern}
An \emph{orthogonal hyperbolic ring pattern} for $\mathscr D$ consists of rings
$(C,c)$ for all vertices of $G$ and points $t\in H^2$ for all
vertices of $G^*$ satisfying the following properties:
\begin{enumerate}[(1)]
 \item The rings associated to incident vertices $v_i$ and $v_j$ in $G$
intersect orthogonally, i.e., the outer circle $C_i$ of one of the vertices 
intersects the inner circle $c_j$ of the other vertex orthogonally and vice 
versa.
\item For any four consecutively neighboring rings associated to the vertices 
$v_{m,n}$, $v_{m+1,n-1}$, $v_{m+2,n}$ and $v_{m+1,n+1}$ of $G$ the 
inner circles $c_{m,n}$ and $c_{m+2,n}$ and the outer circles $C_{m+1,n-1}$ 
and $C_{m+1,n+1}$ intersect at the same point $t_{m+1,n}$, see 
Figure~\ref{fig:Ring}.
\item For any ring $(C_{m,n}, c_{m,n})$ the four touching points 
$t_{m+1,n}= c_{m,n}\cap c_{m+2,n}$, $t_{m,n+1}= C_{m,n}\cap 
C_{m,n+2}$, $t_{m-1,n}= c_{m,n}\cap c_{m-2,n}$ and $t_{m,n-1}= 
C_{m,n}\cap C_{m,n-2}$ have the same orientation, i.e., are in 
counterclockwise order if $r_{m,n}$ is positive and in clockwise order if $r_{m,n}$ is negative.
\end{enumerate}
\end{definition}

An internal ring of a hyperbolic orthogonal ring pattern has four neighboring 
rings in $G$, which consequently touch. This configuration of five rings is 
called a \emph{ring flower}. The existence of a ring flower for every internal 
ring is equivalent to conditions~(1) and (2) of 
Definition~\ref{def:Ringpattern}. Analogously as in~\cite{Bo25} we define 
\emph{generalized orthogonal hyperbolic ring patterns} as sets of orthogonal 
ring patterns for $\mathscr D$, such that every internal ring possesses a ring 
flower. Generalized hyperbolic orthogonal ring patterns do not necessarily 
satisfy condition~(3) of Definition~\ref{def:Ringpattern}.

As explained in~\cite[Sec.~3]{Bo25}, the radii of the rings 
may be expressed in terms of Jacobi elliptic functions.
Due to the hyperbolic Pythagoras' Theorem $\cosh R_j \cosh r_i =\cosh r_j \cosh R_i$ for incident vertices $v_i$ and $v_j$ in $G$. Thus, for every ring the quotient $\cosh r/\cosh R$ is the same and defines a constant $q=\cosh r/\cosh R \leq 1$. Using Jacobi function of modul $q$ (and quater periods $K$ and $K'$), for all vertices of $G$
real variables $\pu$  are introduced by
\begin{align}\label{eq:defbeta}
&\cosh R=\sn (\pu+K+iK', q),\qquad
 \sinh R=i\,\cn (\pu+K+iK', q),\\
 &\sinh r=i\, \dn(\pu+K+iK', q).
\end{align}
Note that~\cite{Bo25} uses the variables $\beta=\pu+K +iK'$ and $\gamma=\pu+K$. As $\sn(\pu+K+iK',q)\geq 1/q$ for real $\pu$
this implies with $q'=\sqrt{1-q^2}$ that
\begin{align}\label{eq:defu}
 \sinh r&=
 -q'\frac{\sn (\pu,q)}{\cn(\pu,q)},\qquad
 &\sinh R&=
 \frac{q'}{q\,\cn(\pu,q)}.
\end{align}

\begin{proposition}[see {\cite[Prop.~3.1 \& Sec.~9]{Bo25}}]
 The formulas~\eqref{eq:defu}
 describe a one to one 
correspondance between the radii $(r,R)\in[-\infty,\infty]\times [R_0,\infty]$ 
of an orthogonal hyperbolic ring pattern with the parameter $q=\cosh r/\cosh R 
\leq 1$ and the variables $\pu$
with $\pu\in[-K,K]$
of Jacobi functions of modulus $q$. Here $R_0=1/q$. We
control the sign by requiring $r = R$ for $q = 1$.
\end{proposition}

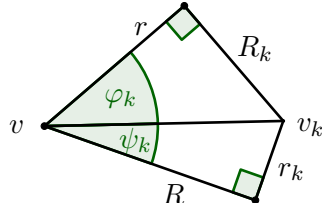
\begin{figure}
\definecolor{wwzzff}{rgb}{0.4,0.6,1}
\definecolor{ffqqqq}{rgb}{1,0,0}
\definecolor{qqwuqq}{rgb}{0,0.39215686274509803,0}
\begin{center}
\begin{tikzpicture}[line cap=round,line join=round,>=triangle 45,x=1cm,y=1cm, 
scale=0.5]
\clip(-13.46,-2.01) rectangle (-1.74,3.73);
\draw[line width=1pt,color=qqwuqq,fill=qqwuqq,fill opacity=0.10000000149011612] 
(-5.209095378584444,-1.3595674553577546) -- 
(-5.80952792322669,-1.148662833942199) -- 
(-6.0204325446422455,-1.7490953785844443) -- (-5.42,-1.96) -- cycle; 
\draw[line width=1pt,color=qqwuqq,fill=qqwuqq,fill opacity=0.10000000149011612] 
(-7.782635053253609,2.765194738014466) -- (-7.367829791268075,2.282559684760858) 
-- (-6.885194738014466,2.697364946746392) -- (-7.3,3.18) -- cycle; 
\draw [shift={(-11,0)},line width=1pt,color=qqwuqq,fill=qqwuqq,fill 
opacity=0.10000000149011612] (0,0) -- (1.2332339718314052:3) arc 
(1.2332339718314052:40.67772294050895:3) -- cycle;
\draw [shift={(-11,0)},line width=1pt,color=qqwuqq,fill=qqwuqq,fill 
opacity=0.10000000149011612] (0,0) -- (-19.354053606636246:3) arc 
(-19.354053606636246:1.2332339718314027:3) -- cycle;
\draw [line width=1pt] (-11,0)-- (-5.42,-1.96);
\draw [line width=1pt] (-11,0)-- (-7.3,3.18);
\draw [line width=1pt] (-7.3,3.18)-- (-4.683780704421436,0.1359712598614212);
\draw [line width=1pt] (-4.683780704421436,0.1359712598614212)-- (-5.42,-1.96);
\draw [line width=1pt] (-11,0)-- (-4.683780704421436,0.1359712598614212);
\draw [fill=black] (-11,0) circle (2.5pt);
\draw[color=black] (-11.73,-0.0255) node {$v$};
\draw [fill=black] (-5.42,-1.96) circle (2.5pt);
\draw[color=black] (-7.59,-1.805) node {$R$};
\draw [fill=black] (-7.3,3.18) circle (2.5pt);
\draw[color=black] (-8.44,2.565) node {$r$};
\draw [fill=black] (-4.683780704421436,0.1359712598614212) circle (1pt);
\draw[color=black] (-3.96,-0.0255) node {$v_k$};
\draw[color=black] (-5.43,2.045) node {$R_k$};
\draw[color=black] (-4.47,-1.205) node {$r_k$};
\draw[color=qqwuqq] (-8.98,0.735) node {$\varphi_k$};
\draw[color=qqwuqq] (-8.59,-0.375) node {$\psi_k$};
\end{tikzpicture}
\end{center}
 \caption{Two orthogonal triangles associated to two neighboring 
rings centered at $v$ and $v_k$.}\label{fig:orthquad}
\end{figure}

Consider an edge $(v,v_k)$ in $G$, see Figure~\ref{fig:orthquad}.
The angles $\varphi_k$ and $\psi_k$ at $v$ in the two right-angled hyperbolic triangles spanned by the 
centers of the two rings and their respective intersection points may be computed by the formulas:
\begin{equation*}
 \varphi_k=\arctan\frac{\tanh R_k}{\sinh r}\qquad \text{and}\qquad 
\psi_k=\arctan\frac{\tanh r_k}{\sinh R}.
\end{equation*}
Using the above parametrization with elliptic functions and the transformations detailed
in~\cite[Sec.~3]{Bo25}, the opening angle $\theta_k=\varphi_k+\psi_k$ may be expressed for $r\not=0$ as
\begin{align*}
 \theta_k &=\arg\left(-\sign(r)   \frac{\sn\left( 
\frac{\pu-\pu_k+iK'}{2}\right)}{\sn\left(K+
\frac{\pu+\pu_k-iK'}{2}\right)}
\right) \\
&= 
\begin{cases}
-\arg\left( 
\sn (K+\frac{\pu+\pu_k+iK'}{2})\, \sn\frac{\pu-\pu_k+iK'}{2}\right)
+\pi & \text{if } r>0,\\[1ex]
-\arg\left( 
\sn (K+\frac{\pu+\pu_k+iK'}{2})\, \sn\frac{\pu-\pu_k+iK'}{2}\right)
 & \text{if } r<0,
\end{cases}
\end{align*}
Introducing for $x\in\R$ the smooth function 
\begin{equation}\label{eq:defg}
 g(x):=\frac{\pi}{2} -\arg\sn\left(\frac{x+iK'}{2}\right) 
=\arctan\frac{(1+q)\sn\frac{x}{2}}{\cn\frac{x}{2}\dn\frac{x}{2}}
\end{equation}
we have (see~\cite[Lemma~3.3]{Bo25}) for two orthogonally 
intersecting hyperbolic rings 
\begin{align}\label{eq:thetag}
  \theta_k= &\begin{cases}
 g(2K+\pu +\pu_k) + g(\pu-\pu_k) &\text{if } r> 0,\\
g(2K+\pu +\pu_k) + g(\pu-\pu_k)- \pi  & \text{if } r< 0.
\end{cases}
\end{align}
For values in the open interval $\pu,\pu_k\in (-K,K)$ we also have
\begin{equation}\label{eq:dtheta}
 \frac{\partial \theta_k}{\partial \pu_k} <0,
\end{equation}
see the Appendix for the calculations. 

Further properties of $g$ are detailed 
in~\cite[Sec.~2]{Bo25}. The 
anti-derivative of $g$ is $F(x)=\int_0^x g(t)dt$,
which is an even convex function with $F(0)=0$, $F'(x)=g(x)$ and 
$F''(x)=(\dn x+q\cn x)/2$.

Applying these representations for the opening angles and their sum  at interior rings, the correspondance between radii and the variables $\pu$ can be used to
acutally characterize generalized hyperbolic orthogonal ring patterns.

\begin{theorem}[{\cite[Theorem~3.2]{Bo25}}]
 Rings build a generalized orthogonal ring pattern in the hyperbolic plane if and only if they are given by the variables 
$\pu\in[-K,K]$
satisfying the following equation for every internal vertex of $G$:
\begin{equation}\label{eq:betaclose}
\prod_{v_k\sim v} \frac{\sn\left( 
\frac{\pu-\pu_k+iK'}{2}\right)}{\sn\left(K+
\frac{\pu+\pu_k-iK'}{2}\right)} = 1.
\end{equation}
\end{theorem}

Using some transformations (see~\cite[Sec.~9]{Bo25} for more details) we obtain 
the following form:
\begin{equation}\label{eq:uclose}
 0= \arg \prod_{k=1}^4\left( 
1-i(1+q) \frac{\sn}{\cn\, \dn}(\frac{\pu_k-\pu}{2}) \right) -\arg
\prod_{k=1}^4\left( 
1-i\frac{q'^2}{1+q}\frac{\sn}{\cn\, \dn}(\frac{\pu_k+\pu}{2}) \right).
\end{equation}

\begin{theorem}[See {\cite[Theorem~3.4 \& Sec.~9]{Bo25}}]
The uniformizing variables 
$\pu$ determine a hyperbolic orthogonal
ring pattern if and only if they lie in the interval 
$\pu\in [-K, K]$
and for all internal vertices satisfy the condition
\begin{equation}\label{eq:condint}
\sum_{k=1}^4 g(2K+\pu +\pu_k) + g(\pu-\pu_k)=2\pi,
\end{equation}
where the sum is taken over all four incident vertices and $g$ is defined
above.

For a boundary vertex $v$ the variables $\pu$ satisfy
\begin{align*}
\sum g(2K+\pu +\pu_k) + g(\pu-\pu_k)&= \Theta(v), & \qquad r>0, \\
 \sum  g(2K+\pu +\pu_k) + g(\pu-\pu_k)&= \Theta(v)+\pi {\mathcal V}(v), & \qquad
r<0,
\end{align*}
where $\Theta(v)$ is the total nominal angle at the vertex (positive for $r > 0$
and negative for $r < 0$), and ${\mathcal V} (v)$ is the valence of $v$, i.e.\ 
the number 
of neighboring rings orthogonal to the ring centered at $v$.
\end{theorem}

\begin{remark}\label{rem:arctan}
 For the goal of convergence, we are mostly interested in the case
of non-negative radii of the inner circles
corresponding to $\pu\in[-K,0]$. (The case of non-positive radii is analogous.) Using transformations as in~\cite[Section~9]{Bo25}, we deduce that
\begin{align*}
 \theta_k&= 
 g(2K+\pu +\pu_k) + g(\pu-\pu_k) \\
 &= \pi -\arg\left( 1-i\frac{q'^2}{1+q}\frac{\sn}{\cn\, 
\dn}(\frac{\pu_k+\pu}{2}) \right)  -\arg i\left(
1+i(1+q) \frac{\sn}{\cn\, \dn}(\frac{\pu_k-\pu}{2})
\right) \\
&= \frac{\pi}{2} -\arg\left( 1-i\frac{q'^2}{1+q}\frac{\sn}{\cn\, 
\dn}(\frac{\pu_k+\pu}{2}) \right)  +\arg \left(
1-i(1+q) \frac{\sn}{\cn\, \dn}(\frac{\pu_k-\pu}{2}) \right)
\end{align*}
Thus we can rewrite the left hand side of equation~\eqref{eq:condint} as
\begin{multline}
\sum_{k=1}^4 (g(2K+\pu +\pu_k) + g(\pu-\pu_k))-2\pi \\
=\arg \prod_{k=1}^4\left( 
1-i(1+q) \frac{\sn}{\cn\, \dn}(\frac{\pu_k-\pu}{2})
\right) -\arg \prod_{k=1}^4\left( 
1-i\frac{q'^2}{1+q}\frac{\sn}{\cn\, 
\dn}(\frac{\pu_k+\pu}{2}) \right).
\end{multline}
Using that
\begin{align*}
 \arg \prod_{k=1}^4\left( 1-i\frac{q'^2}{1+q}\frac{\sn}{\cn\, 
\dn}(\frac{\pu_k+\pu}{2}) \right) &=
-\sum_{k=1}^4 \arctan\frac{q'^2}{1+q}\frac{\sn}{\cn\, 
\dn}(\frac{\pu_k+\pu}{2}), \\
\arg \prod_{k=1}^4\left( 
1-i(1+q) \frac{\sn}{\cn\, \dn}(\frac{\pu_k-\pu}{2}) \right)
&= -\sum_{k=1}^4 \arctan (1+q) \frac{\sn}{\cn\, 
\dn}(\frac{\pu_k-\pu}{2}),
\end{align*}
we obtain 
\begin{multline}\label{eq:sumarctan}
\sum_{k=1}^4 (g(2K+\pu +\pu_k) + g(\pu-\pu_k))-2\pi \\
=\sum_{k=1}^4 \arctan\frac{q'^2}{1+q}\frac{\sn}{\cn\, 
\dn}(\frac{\pu_k+\pu}{2})
-\sum_{k=1}^4 \arctan (1+q)\frac{\sn}{\cn\, 
\dn}(\frac{\pu_k-\pu}{2}).
\end{multline}
\end{remark}

For the variational description of hyperbolic orthogonal ring patterns, we 
consider the functional
\begin{equation}
 S(\pu)=\sum_{v_j\sim v_k} (F(\pu_j-\pu_k)+F(\pu_j+\pu_k+2K))+ \sum_{v_j} \Phi_j \pu_j,
\end{equation}
where the first sum is taken over all egdes with incident vertices $v_j$ and 
$v_k$ and the second sum over all vertices of $G$. The parameters $\Phi_j$ are 
prescribed as follows:
\begin{equation}\label{eq:defPhi}
 \Phi_j=\begin{cases}
         -2\pi & \text{ for inner rings,} \\
          -\Theta(v_j)  & \text{ for positively oriented boundary 
rings,} \\
  -\pi {\mathcal V}(v_j)-\Theta(v_j)  & \text{ for negatively oriented boundary 
rings.}
        \end{cases}
\end{equation}
Here $\Theta(v)$ is a given parameter (cone angle) at boundary vertices and 
${\mathcal V} (v)$ is the valence of $v$. For further use, note that
\begin{equation}\label{eq:partialS}
 \frac{\partial S}{\partial \pu_j}(\pu)= \sum_{v_j\sim v_k}
(g(\pu_j-\pu_k)+g(\pu_j+\pu_k+2K))+  \Phi_j.
\end{equation}

We denote the set of boundary 
vertices by $\partial V$.

\begin{theorem}[{\cite[Theorems~5.2 and 5.3 \& 
Sec.~9]{Bo25}}]\label{theo:convex}
\begin{enumerate}[a)]
 \item 
 Let $\Theta:\partial V\to (-2\pi, 2\pi)$ be the prescribed cone angles at the 
boundary vertices. Assume that $|\Theta (v_j)|<\pi$ for the boundary vertices 
of valence ${\mathcal V} (v)=1$. Then the critical points of the functional $S$ 
with $\Phi$ defined in~\eqref{eq:defPhi} correspond to hyperbolic orthogonal 
ring patterns.

In particular, if $\pu$ is a critical point of $S$, then all $\pu\in[-K,K]$.
\item  The functional $S$ is convex.
\end{enumerate}
\end{theorem}

Based on the convex functional $S$, hyperbolic orthogonal ring patterns can 
be constructed from boundary values due to the following theorem.

\begin{theorem}[{\cite[Theorem~6.1 \& Sec.~9]{Bo25}}, Dirichlet boundary value 
problem]\label{theo:unique}
 For any choice of prescribed $\pu:\partial V\to [-K,K]$ for boundary rings,
there exists a unique hyperbolic orthogonal ring pattern with the corresponding 
boundary radii given by~\eqref{eq:defu}.
\end{theorem}

\section{Convergence of variables $u^\eps$ to given solution $\u$ of sinh-Gordon 
equation}\label{Sec:Convu}

In this section we formulate and prove our convergence result for the $u$-variables.

\begin{theorem}\label{theoConv}
Let $\u:D\to(-\infty,0)$ be a smooth map satisfying the sinh-Gordon equation $\Delta \u-\sinh(2\u)=0$. Let 
$B\subset D$ be a compact set which is the closure of its simply connected 
interior $B_{int}$ and assume that the boundary $\partial B$ is smooth. Assume
for simplicity that $0\in B_{int}$.

For each $\eps>0$ let ${\mathscr D}^\eps_B$ be a 
be a subcomplex of the $\Z^2$-lattice scaled by $\eps$ whose support is contained in $B$ and is
homeomorphic to a closed disc. We further assume that $0$ is an 
interior vertex of ${\mathscr D}^\eps_B$. Let $G^\eps_B$ and $(G^\eps_B)^*$ be the corresponding dual graphs.

Set $q'=\eps$ and let $q=\sqrt{1-\eps^2}$ be the modul of Jacobi elliptic
functions. Let $\eps_0>0$ be such that the corresponding quater-period $K_0$ satisfies $K_0(\sqrt{1-\eps_0^2})> \sup_{x\in B}|\u(x)|$.

Then if $\eps_0>\eps>0$ is small enough (depending on $B$ and $\u$), the 
following holds.
\begin{enumerate}[(i)]
\item There exists a hyperbolic ring pattern for ${\mathscr D}^\eps_B$ 
with associated discrete map
$u^\eps$ on $G^\eps_B$ which satisfies equation~\eqref{eq:uclose} at all 
interior vertices and
\begin{equation}\label{eq:boundu}
 u^\eps(v)=\u(v)\qquad \text{for all boundary vertices } v \text{ of } 
G^\eps_B.
\end{equation}

Furthermore, $u^\eps$ approximates $\u$ uniformly on $B$ with error 
of order~$\eps^2$:
\begin{equation}\label{eq:convu}
 \left|u^\eps(v)-\u(v)\right|\leq C\eps^2
\end{equation}
holds for all vertices $v$ of  $G^\eps_B$, where the constant $C$ depends only on $B$ and $\u$, but not on $v$.
\item
If the subcomplexes ${\mathscr D}^\eps_B$ are be chosen such that they 
approximate the compact set $B$ for $\eps\to0$, the discrete maps $u^\eps$ 
converge in $C^\infty(B_{int})$ to $\u$.
\end{enumerate}
\end{theorem}

Note that the subcomplexes ${\mathscr D}^\eps_B$ may be chosen such that they exhaust the compact set $B$.

For the rest of this section, we always assume that $\u$, $B$, $q$ and $q'$ are 
given as in Theorem~\ref{theoConv}. Also, we assume that $0<\eps<\eps_0$ and we will further reduce the upper bound as necessary.

\subsection{Proof of $C^1$-convergence of $u^\eps$ to $u$}

Our strategy for the proof follows ideas from~\cite{Bue08} and \cite{Bue16}
which are
adapted to the case of  hyperbolic ring patterns. Existence and
uniqueness of the solution $u^\eps$ for the given boundary value problem
follow directly from Theorem~\ref{theo:unique}. As the smooth given solution $\u$  has
the right boundary values by construction, the main idea for our convergence
proof is based on the fact that $\u$ nearly satisfies the discrete closing
condition~\eqref{eq:uclose} at interior vertices. This is a consequence of the
symmetric structure of the underlying $\Z^2$-lattice.

\begin{lemma}[Compare {\cite[Sec.~9]{Bo25}}]\label{lem:Tayloru}
 Let $\u:D\to(-\infty,0)$ be a smooth map satisfying $\Delta \u-\sinh(2\u)=0$.
 Define $\eps_0$, $q'=\eps$, $q=\sqrt{1-\eps^2}$ and $G^\eps_B$
Denote by $u_{G^\eps}$ the restriction
of $\u$ to the vertices of $G^\eps_B$.
Then for $\eps<\eps_0$ small enough and for all 
interior vertices $v$ with neighbors $v_1,\dots,v_4$ we have
\begin{equation*}
 \left| \sum_{k=1}^4 (g(2K+u_{G^\eps}(v) +u_{G^\eps}(v_k)) + g(u_{G^\eps}(v)-u_{G^\eps}(v_k)))-2\pi \right| \leq 
C_1\eps^4,
\end{equation*}
where the constant $C_1$ only depends on $\u$ and $B$.
\end{lemma}
\begin{proof}
As $\u$ is smooth, the differences $u_{G^\eps}(v_k)-u_{G^\eps}(v)$ for 
incident vertices are of order $\eps$.
Using the representations in Remark~\ref{rem:arctan}, 
especially~\eqref{eq:sumarctan} and a Taylor expansion, we now deduce that for 
$\eps=q'$ small enough
\begin{align*}
 &\sum_{k=1}^4 \arctan (1+q) \frac{\sn}{\cn\, 
\dn}(\frac{u_{G^\eps}(v_k)-u_{G^\eps}(v)}{2}) -\sum_{k=1}^4 
\arctan\frac{q'^2}{1+q}\frac{\sn}{\cn\, 
\dn}(\frac{u_{G^\eps}(v_k)+u_{G^\eps}(v)}{2}) \\
&= \eps^2(\Delta \u(v) -\sinh(2\u(v))) + \O(\eps^4).
\end{align*}
As  $\Delta \u-\sinh(2\u)=0$ by assumption, the claim follows.
\end{proof}

As in~\cite{Bue16}, one key observation for the proof of Theorem~\ref{theoConv} is that we can control the sign of
the leading term in the above Taylor expansion if we replace $u_{G^\eps}$ by 
\begin{align*}
w^\pm &= u_{G^\eps} +\begin{cases}\pm\eps^2C & 
\text{for interior vertices}, \\
        0 & \text{for boundary vertices},
       \end{cases}
\end{align*}
where $C>C_1/32$ is a positive constant. 

\begin{lemma}\label{lemC}
 There exists a constant $C>0$, depending only on $\u$ and its derivatives, such 
that for $\eps$ small enough and all interior vertices $v_0$ there holds
\begin{align*}
&\sum_{v_0\sim v_k} (g(w^+_0-w^+_k)+g(w^+_0+w^+_k+2K))-2\pi >0, \\
&\sum_{v_0\sim v_k} (g(w^-_0-w^-_k)+g(w^-_0+w^-_k+2K))-2\pi <0.
\end{align*}
\end{lemma}
\begin{proof}
As in the proof of Lemma~\ref{lem:Tayloru}, we again use 
representation~\eqref{eq:sumarctan} of Remark~\ref{rem:arctan}.
Using a Taylor expansion for interior vertices whose 
neigbours are all interior vertices $v_0$, we obtain:
\begin{align*}
 & \sum_{k=1}^4 
\arctan\frac{q'^2}{1+q}\frac{\sn}{\cn\, 
\dn}(\frac{w^\pm(v_k)+w^\pm(v_0)}{2})-\sum_{k=1}^4 \arctan (1+q) 
\frac{\sn}{\cn\, 
\dn}(\frac{w^\pm(v_k)-w^\pm(v_0)}{2}) \\
&= \pm 2C\cosh(2u(v_0))\eps^2 + \O(\eps^4).
\end{align*}
For  interior vertices $v_0$ who have $b$ incident boundary vertices as 
neigbours with $b\geq  1$ we similarly deduce from a Taylor expansion that
\begin{align*}
 &\sum_{k=1}^4 
\arctan\frac{q'^2}{1+q}\frac{\sn}{\cn\, 
\dn}(\frac{w^\pm(v_k)+w^\pm(v_0)}{2})-\sum_{k=1}^4 \arctan (1+q) 
\frac{\sn}{\cn\, 
\dn}(\frac{w^\pm(v_k)-w^\pm(v_0)}{2})  \\
&= \pm C(b+ \cosh(2u(v_0))(2-b/4))\eps^2 + \O(\eps^4).
\end{align*}
Thus if $\eps$ is small enough, the claim follows.
\end{proof}

As a consequence of this lemma, there exist positive constants 
$C$ such that
all components of the gradient of $S$ are positive for $w^+$ and negative for 
$w^-$. Therefore, we interpret $w^\pm$ as analogues of superharmonic and 
subharmonic functions, respectively. In the remaining proof we show that for 
small enough $\eps>0$ the solution $u^\eps$ is the unique minimizer of $S$ 
in the $n$-dimensional interval
 \begin{align*}
 W^\eps=\{\eta:V^\eps_B\to\R\ |\ & \eta(v) = u_{G^\eps}(v) \text{ for all } v\in\partial V^\eps_B, \\
 &\,w^-(v)\leq \eta(v)\leq w^+(v) \text{ for all } v\in V^\eps_{B, \text{int}} \},
\end{align*}
where $n$ is number of vertices of $V^\eps_B$. Note that the restriction $u_{G^\eps}$ of the smooth solution $\u$ to the lattice is by definition contained in $W^\eps$.
Furthermore, all functions in $W^\eps$ satisfy $\eta(v_j)-\eta(v_i)= \O(\eps)$ for incident vertices $v_j\sim v_i$.

By Theorem~\ref{theo:convex}, the functional $S$ is convex and therefore 
attains its minimum in the closed set 
$W^\eps$. It remains to show that the minimizer does not lie on the boundary 
which is mainly due to the monotonicity of the angle function~\eqref{eq:dtheta}.

\begin{lemma}\label{lem:gradS}
 On the boundary of $W^\eps$, the negative gradient $-\text{grad}\, S$ always 
points to the interior of $W^\eps$. Therefore, $S$ cannot attain its minimum on $W^\eps$ at a boundary point.
\end{lemma}
\begin{proof}
 For any interior vertex $v_i$, consider the corresponding boundary face $W_i^+ 
=\{\eta\in W^\eps\ |\ \eta(v_i) = w^+(v_i)\}$. Analogous considerations hold 
for $W_i^-$. 

By~\eqref{eq:partialS} and \eqref{eq:defPhi} we deduce for the partial 
derivative $\frac{\partial S}{\partial \eta_i}$ at the boundary $W^+_i$:
\begin{align*}
 \frac{\partial S}{\partial \eta_i}(\eta)&= \sum_{v_i\sim v_k} 
(g(\eta_i-\eta_k)+g(\eta_i+\eta_k+2K))-2\pi \\
& =\sum_{v_i\sim v_k} 
(g(w^+_i-\eta_k)+g(w^+_i+\eta_k+2K))-2\pi.
\end{align*}

Note that the function $x\mapsto g(w^+_i-x)+g(w^+_i+x+2K)$ is strictly 
monotonically decreasing in $x$ due to~\eqref{eq:dtheta}.
As $\eta_k\leq w^+_k$ we thus deduce that
\begin{align*}
 \frac{\partial S}{\partial \eta_i}(\eta)\geq \sum_{v_i\sim v_k} 
(g(w^+_i-w^+_k)+g(w^+_i+w^+_k+2K))-2\pi >0
\end{align*}
where the inequality is strict if $\eta_k\neq w^+_k$ for one of the 
incident vertices of $v_i$ and we have applied Lemma~\ref{lemC} for the last
inequality.
\end{proof}

\begin{proof}[Proof of Theorem~\ref{theoConv}, Part (i)]
 The existence of a hyperbolic orthogonal ring pattern with given boundary 
values, and thus of a corresponding discrete function $u^\eps$, follows 
directly from Theorem~\ref{theo:unique}. By Theorem~\ref{theo:convex}, $u^\eps$ 
is the unique minimizer of the convex functional $S$. Choose $\eps$ small 
enough such that the estimates in Lemmas~\ref{lem:Tayloru} and \ref{lemC} 
hold. Then by Lemma~\ref{lem:gradS} the negative gradient of $S$ points into 
the interior of $W^\eps$. Consequently, $S$ must attain its minimum within 
$W^\eps$. By construction of $W^\eps$, this proves the a priori 
estimate~\eqref{eq:convu}. 
\end{proof}

\subsection{Proof of $C^\infty$-convergence of $u^\eps$ to $u$}

Our strategy is inspired by methods of the proof of
$C^\infty$-con\-ver\-gence for hexagonal circle packings in~\cite{HeSch98}, 
which were already applied for the convergence of (orthogonal) circle patterns in~\cite{Bue08} and for triangular lattices in
in~\cite{Bue16}. Thanks to the regular structure of the $\Z^2$-lattice, we introduce for any function on the vertices
$\eta:V^\eps_B\to\R$ the \emph{discrete directional derivative}
$\partial_k^\eps\eta$ defined for $k=0,\dots,3$ on all interior vertices 
$V^\eps_{int}$ by
\begin{equation*}
 \partial_k^\eps \eta = (\eta(v+\eps i^k)-\eta(v))/\eps.
\end{equation*}
Also, the \emph{discrete Laplacian} of $\eta$ is defined by the usual formula
\begin{equation*}
 \Delta^\eps \eta (v)={\frac{1}{\eps^2}} \sum_{k=0}^3 (\eta(v+\eps 
i^k)-\eta(v)) ={\frac{1}{\eps}} \sum_{k=0}^3 \partial_k^\eps \eta.
\end{equation*}
Note that $\partial_k^\eps$ and $\Delta^\eps$ commute with each other.

We  denote the $L^\infty(Y)$-norm for any subset $Y\subset V^\eps$ by
$\|\eta\|_Y=\sup_{v\in Y}|\eta(v)|$.
Furthermore, to any subset of vertices $Y$ of $V^\eps$ set $Y^{(0)}=Y$ and for 
all $j\geq 1$ 
denote successively by $Y^{(j)}$ the set of interior vertices of $Y^{(j-1)}$.
Then for any function $\eta:Y\to\R$ discrete partial derivatives
$\partial_{k_j}^\eps\partial_{k_{j-1}}^\eps \dots \partial_{k_1}^\eps \eta$ 
can be defined on $Y^{(j)}$ if $Y^{(j)}$ is non-empty.

We will show inductively, that
these discrete partial derivatives of $u^\eps$ converge locally uniformly, that 
is uniformly on any given compact subset of the interior of $B$ to 
the corresponding partial derivatives of the smooth given function $\u$. If these convergences hold for
all $j\in\N$, we say that the {\it convergence is in $C^\infty$}.

The following lemma
states that it is sufficient to show that all discrete partial derivatives of 
$u^\eps$ are locally uniformly bounded, as we can identify the limit of any 
convergent subsequence with $u_{G^\eps}$.

\begin{lemma}[{\cite[Lemma~2.1]{HeSch98}}]\label{lem:bound}
 Let $n$ be a positive integer. Let $Y^\eps=V^\eps_B\cap K$ for a given compact
set $K$ and let $\eta^\eps:Y^\eps \to \R$. Suppose that the discrete partial 
derivatives 
$\partial_{k_j}^\eps\partial_{k_{j-1}}^\eps \dots \partial_{k_1}^\eps 
\eta^\eps$ are uniformly bounded, i.e.
\begin{equation*}
 \|\partial_{k_j}^\eps\partial_{k_{j-1}}^\eps \dots \partial_{k_1}^\eps 
\eta^\eps \|_{Y^{(j)}} \leq C_j,
\end{equation*}
for all $j=1,\dots,n$. Then there is a $C^{n-1}$-function $h$ and a subsequence 
$\eps_m\to 0$ such that $\eta^{\eps_m}\to h$ along that subsequence.
\end{lemma}

Recall that $u^\eps$ and $\partial_k^\eps u^\eps$ 
are uniformly bounded on $B$ for all $\eps$ small enough by~\eqref{eq:convu}.
Further, for all discrete functions $\eta^\eps$ which are uniformly bounded 
together 
with their discrete directional derivatives $\partial_k^\eps \eta^\eps$,
representation~\eqref{eq:sumarctan} implies that
\begin{align*}
  0&=\sum_{k=1}^4 (g(2K+\eta^\eps +\eta_k^\eps) + g(\eta^\eps-\eta_k^\eps))-2\pi 
\\
  &=\sum_{k=1}^4 \sum_{j=0}^\infty 
\frac{h_1^{(j)}(\eta)}{j!} \left(\frac{\eps\partial_k^\eps 
\eta^\eps}{2}\right)^j
-\sum_{k=1}^4 \sum_{j=2}^\infty 
\frac{h_2^{(j)}(0)}{j!} \left(\frac{\eps\partial_k^\eps \eta^\eps}{2}\right)^j
-\frac{1+q}{2}\sum_{k=1}^4 \eps\partial_k^\eps \eta^\eps
\end{align*}
where we have used the smoothness of the functions 
\begin{align*}
 h_1(x)&=  \arctan\frac{q'^2}{1+q}\frac{\sn}{\cn\,\dn}(\frac{x}{2}),
 &\text{and} &&
  h_2(x)&= \arctan(1+q)\frac{\sn}{\cn\,\dn}(\frac{x}{2}).
\end{align*}
Recall that $q=\sqrt{1-\eps^2}$ and $q'=\eps$, thus $h_1$ and $h_2$ also 
depend 
smoothly on $\eps$. Furthermore, $h_1/\eps^2$ can be continued smoothly at 
$\eps=0$, therefore we deduce that
\begin{align}
 \Delta^\eps \eta^\eps &= 
\sum_{k=1}^4 \sum_{j=0}^\infty 
 \frac{1}{\eps^2}\frac{h_1^{(j)}(\eta^\eps)}{j!} 
\left(\frac{\eps\partial_k^\eps 
\eta^\eps}{2}\right)^j
-\sum_{k=1}^4 \sum_{j=2}^\infty 
\frac{h_2^{(j)}(0)}{j!} 
\left(\frac{\partial_k^\eps \eta^\eps}{2}\right)^j\eps^{j-2} \notag \\
&\qquad +\frac{\eps}{2(\sqrt{1-\eps^2}+1)}\sum_{k=1}^4
\partial_k^\eps \eta^\eps \notag \\
&=: F(\eta^\eps,\partial_0^\eps \eta^\eps,\dots,\partial_3^\eps \eta^\eps,\eps) 
\label{eq:defF}
\end{align}
defines a smooth function $F$ depending on $\eps$ and on the values of 
$\eta^\eps$ and its discrete directional derivatives $\partial_k^\eps \eta^\eps$.
Now the following Regularity Lemma can be applied inductively in order to 
obtain the boundedness of the discrete directional derivatives of all orders of 
$u^\eps$.

\begin{lemma}[Regularity Lemma, {\cite[Lemma 7.1]{HeSch98}}]\label{lem:Reg}
 Let $W\subset V^\eps_B$ and let $\eta:W\to \R$ be any function.
There are constants $C_5,C_6>0$, independent of $W$ and $u$, such that
\begin{equation}
 |\eta(v_0)-\eta(v_1)|\rho \leq C_5\|\eta\|_W +\rho^2 C_6 
\|\Delta\eta\|_{W^{(1)}}
\end{equation}
for all vertices $v_1\in W$ incident to $v_0$, where $\rho$ is the Euclidean
distance of $v_0$ to the boundary $W_\partial$.
\end{lemma}

\begin{proof}[Proof of Theorem~\ref{theoConv}, Part (ii)]
 By Lemma~\ref{lem:bound} it is sufficient to show, that all discrete 
directional derivatives of $u^\eps$ are uniformly bounded, as we have already 
identified the (smooth) limit function with the given function $u$.

We establish these bounds inductively. 
$u^\eps$ and $\partial_k^\eps u^\eps$ 
are uniformly bounded on $B$ by estimate~\eqref{eq:convu}, thus the claim holds 
for $n=0$ and $n=1$. Now assume that all discrete 
directional derivatives of $u^\eps$ are uniformly bounded up to order $n$. 
Recall that 
\begin{equation*}
 \Delta^\eps \partial_{k_j}^\eps\partial_{k_{j-1}}^\eps \dots 
\partial_{k_1}^\eps u^\eps 
 \partial_{k_j}^\eps\partial_{k_{j-1}}^\eps \dots 
\partial_{k_1}^\eps \Delta^\eps u^\eps = 
\partial_{k_j}^\eps\partial_{k_{j-1}}^\eps \dots 
\partial_{k_1}^\eps F(\eta,\partial_0^\eps \eta,\dots,\partial_3^\eps 
\eta,\eps),
\end{equation*}
where~\eqref{eq:defF} is used for the last equality. As $F$ is a smooth 
function in all its variables, the last expression is 
bounded by the induction hypothesis. Thus the induction step follows from the 
Regularity Lemma~\ref{lem:Reg}. This completes the proof.
\end{proof}

\section{Convergence of hyperbolic ring patterns}\label{Sec:ConvR}

The convergence of the parameters $u^\eps$ to the smooth given function $\u$ by Theorem~\ref{theoConv}
implies that also the hyperbolic ring patterns converge if they are suitably normalized.

To this end, first
assume that we have fixed (by a Moebius transformation) the point corresponding to the vertex $0$ at
the origin and one fixed neighbor of this vertex on the positive real axis. Furthermore, applying the approximation result of Theorem~\ref{theoConv} to~\eqref{eq:defu} and using  $\sn(x,q)\to \tanh x$ and $\cn(x,q)\to 1/\cosh x$ for $q\to 1$, we deduce that
\begin{align*}
 \frac{r_{m,n}}{\eps} &=\sinh \u(v_{m,n}) + \O(\eps^2) \\
  \frac{R_{m,n}}{\eps} &=\cosh \u(v_{m,n}) + \O(\eps^2) \\
   \frac{r_{m,n}}{R_{m,n}} &=\tanh \u(v_{m,n}) + \O(\eps^2) \\
   (\theta_k-\frac{\pi}{2})/\eps &=-\partial \u + \O(\eps)
\end{align*}
where the partial derivative $\partial\u$ is taken in the direction of $v_k-v$.

Remember that the touching circles meet by definition
orthogonally at all touching points.
Therefore, we can build
partial derivatives of the position function of the hyperbolic ring patterns between the midpoints of touching circles in both directions (inner circles in one direction and outer circles in the orthogonal direction). The above estimates imply that all these partial derivatives converge.

Furthermore, we deduce that the image domain of the ring patterns is bounded. By reducing to a
suitable subsequence we also deduce that for any compact subdomain of the 
preimage, the hyperbolic ring patterns converge to a continuous function $h$. 
The limit function is also differentiable as the discrete partial derivatives 
also converge. Moreover, using the orthogonality at the touching points 
$t_{m,n}$ and the the above limit of $\frac{r_{m,n}}{R_{m,n}} $, we deduce that
\begin{align}\label{eq:harmonic}
 \frac{-\partial_y h}{\partial_x h} &=\tanh \u, &\text{thus }\qquad
\frac{\partial_{\bar z}h}{\partial_z h} = \text{e}^{-2\u},
\end{align}
where we have defined the $x$-axis parallel to the touching direction of the outer
$R$-circles. Therefore, the limit function $h$ is a harmonic map in the hyperbolic plane. As this holds true for all subsequences, we see that the whole sequence converges.

\begin{remark}
The limit function $h$, which is a harmonic map in the hyperbolic plane by~\eqref{eq:harmonic}, can
be interpreted as the Gauss-map of a spacelike CMC-surface in Lorentz space, see for example~\cite{BJS19}. The given function $\u$ represents the conformal metric $\text{e}^{\u}$.
As discrete CMC-surfaces in Lorentz space can be constructed using hyperbolic orthogonal ring patterns as detailed in~\cite{BHS24}, the approximation of the  a hyperbolic map by hyperbolic orthogonal ring patterns also leads to the convergence of the discrete CMC-surfaces after suitable normaizations.
\end{remark}

\appendix

\section{Proof of inequality~\eqref{eq:dtheta}}
By~\cite[{Eq.~(15)}]{Bo25} the derivative of the function $g$ given 
in~\eqref{eq:defg} can be rewritten as
\begin{equation*}
 g'(x)=\frac{1}{2}(\dn x+q\cn x).
\end{equation*}
Applying~\eqref{eq:thetag} we deduce that
\begin{align}
 \frac{\partial \theta_k}{\partial u_k} &= g'(2K+u+u_k)-g'(u-u_k) \\
  &=\frac{1}{2}\left( \dn(2K+u+u_k)+q\cn(2K+u+u_k)-\dn(u-u_k)-q\cn(u-u_k) 
\right) \notag\\
&=-\frac{q^2 \sn(K+u)\cn(K+u)\sn(K+u_k)\cn(K+u_k)}{1-q^2 \sn^2(K+u) 
\sn^2(K+u_k)} \notag \\
&\quad -\frac{q\sn(K+u)\dn(K+u)\sn(K+u_k)\dn(K+u_k)}{1-q^2 \sn^2(K+u) 
\sn^2(K+u_k)} \notag \\
&=-\frac{q \sn(K+u)\sn(K+u_k)(q\cn(K+u)\cn(K+u_k) 
+\dn(K+u)\dn(K+u_k)}{1-q^2 \sn^2(K+u) \sn^2(K+u_k)} ,\label{eq:dt}
\end{align}
where we have used the addition formulas for $\dn$ and $\cn$.

Note that $\sn(K+u)\geq 0$ and $\sn(K+u_k)\geq 0$ due to $u,u_k\in[-K,K]$ and 
these inequalities are strict if $u$ and $u_k$ belong to the open interval 
$(-K,K)$. As $\dn x\geq |\cn x|$ the enumerator 
and the denominator of~\eqref{eq:dt} are both positive for $0<q<1$. Thus 
$\frac{\partial \theta_k}{\partial u_k}<0$.

\bibliographystyle{amsalpha} 
\bibliography{ringpatterns.bib}   

\end{document}